\documentclass[a4paper,twoside]{amsart}
\usepackage{amssymb}
\usepackage{amsmath}
\usepackage{amsthm}
\usepackage[non-compressed-cites]{amsrefs}

\usepackage{mathrsfs}
\usepackage{bbm}

\usepackage[normalem]{ulem}

\usepackage{enumerate}

\usepackage{tikz}
\usetikzlibrary{shapes,arrows}
\usepackage{float}
\usepackage{pgfmath}
\usepackage{color}

\newcommand{\bluecomment}[1]{{#1}}
\usepackage{graphicx}
\usetikzlibrary{positioning,patterns,calc}
\usepackage{caption}
\usepackage{subcaption}

\usepackage{comment}

\newcommand{\real}{\mathbb{R}} 
\newcommand{\pinteger}{\mathscr{P}} 

\newcommand{\varifolds}{\mathbf{V}} 
\newcommand{\rvarifolds}{\mathbf{RV}} 
\newcommand{\ivarifolds}{\mathbf{IV}} 
\newcommand{\cvarifolds}{\mathbf{CV}} 
\newcommand{\cbvarifolds}{\mathbf{AV}} 
\newcommand{\restrict}{\mathop{\llcorner}} 
\newcommand{\graphv}[1]{\mathbf{{\upsilon}}(#1)} 
\newcommand{\boundary}{\mathcal{B}}

\newcommand{\smoothcompact}[2]{\mathcal{D}(#1,#2)} 
\newcommand{\compactfunctions}[2]{\mathcal{K}(#1,#2)} 
\newcommand{\conecompact}[1]{\mathcal{C}_c^1(#1)} 
\newcommand{\lploc}[1]{\mathbf{L}_{\textup{loc}}^{#1}} 

\newcommand{\graph}[1]{\textup{graph}(#1)} 
\newcommand{\homomorphism}[2]{\textup{Hom}(#1,#2)} 

\newcommand{\openball}[2]{\mathbf{{U}}\left(#1,#2\right)} 
\newcommand{\closedball}[2]{\mathbf{{B}}\left(#1,#2\right)} 
\newcommand{\grassmannian}[2]{\mathbf{{G}}(#1,#2)} 
\newcommand{\totalgrassmannian}[2]{\mathbf{{G}}_{#1}(#2)} 

\newcommand{\tangentmap}[3]{\textup{Tan}^{#1}(#2,#3)} 
\newcommand{\tangent}[1]{\textup{Tan}(#1)} 
\newcommand{\rot}{R_{\frac{2\pi}{3}}} 
\newcommand{\reflect}{r_{x_1x_3}} 

\newcommand{\hausdorff}{\mathcal{H}} 
\newcommand{\mass}[1]{\|#1\|} 
\newcommand{\support}[1]{\textup{supp}\,#1} 
\newcommand{\lebesgue}[1]{\mathcal{L}^{#1}} 

\newcommand{\integrald}{\textup{d}} 
\newcommand{\innerproduct}[2]{\langle\, #1,#2\,\rangle} 
\newcommand{\cardinality}[1]{\textup{card }#1} 

\usepackage{etoolbox}
\makeatletter
\patchcmd{\@maketitle}
  {\ifx\@empty\@dedicatory}
  {\ifx\@empty\@date \else {\vskip3ex \centering\footnotesize\@date\par\vskip1ex}\fi
   \ifx\@empty\@dedicatory}
  {}{}
\patchcmd{\@adminfootnotes}
  {\ifx\@empty\@date\else \@footnotetext{\@setdate}\fi}
  {}{}{}
\makeatother

\newtheorem{theorem}{Theorem}[section]
\newtheorem{lemma}[theorem]{Lemma}
\newtheorem{remark}[theorem]{Remark}

\newtheorem{corollary}[theorem]{Corollary}

\newtheorem{definition}[theorem]{Definition}

\newtheorem*{theorem*}{Theorem}

\newtheorem*{corollary*}{Corollary}

\newtheorem{claim}{\texttt{Claim}}
\makeatletter
\@addtoreset{claim}{theorem}
\makeatother

\begin{document}
\title[Varifold Without Curvature Decomposition]{A Curvature Varifold Whose Weak Second Fundamental Form is not Preserved Under Decompositions}
\author{Nicolau S. Aiex
}
\date{\today}
\address{88, Sec.4, Tingzhou Road, SE Building SE809, Taipei, 116059, Taiwan}
\email{nsarquis@math.ntnu.edu.tw}

\begin{abstract}
We construct a curvature varifold that does not admit a decomposition whose components are curvature varifolds.
\end{abstract}

\maketitle

\section{Introduction}
The notion of a curvature varifold was introduced by Hutchinson in \cite{hutchinson1986} and it describes a weak version of second fundamental form on varifolds.
This allows to use the theory of varifolds to study geometric functionals that involve curvature, see for example \cite{mondino2014}.
Furthermore, Hutchinson also proved a regularity result when the curvature is sufficiently integrable see \cite{hutchinson1986.2} and \cite{aiex2024:arxiv} for the complete proof.

The concept of weakly differentiable functions on varifolds introduced by Menne in \cite{menne2016.1} establishes the connection between the weak second fundamental form and the tangent map of the varifold that is naturally defined.
In fact, Menne \cite{menne2016.1}*{Theorem 15.6} proves in particular that the weak second fundamental form corresponds to the weak derivative of the tangent map.

The theory of weakly differentiable functions and the properties developed in \cite{menne2016.1} were essential to complete the proof of graphical representation of curvature varifolds in \cite{aiex2024:arxiv}.
An important part of the proof was to construct partitions of the varifold at small scale in which every element of the partition is a curvature varifold and carries tilt-excess estimates at all scales.

These notes will, in some sense, show that the above is optimal and it cannot be improved to construct a decomposition instead of a partition.
The difference between a decomposition and a partition is that in the former one requires the elements to be indecomposable, that is, they cannot be further separated into pieces and are essentially connected.

As one would expect, regular varifolds are curvature varifolds with respect to the usual second fundamental form.
A simple example of a varifold that is not a curvature varifold (see Remark \ref{remark not curvature}) is a union of $3$ half-planes meeting along a common line at $120^\circ$ angle.
However, the varifold given by a union of $3$ planes intersecting along a common line at $60^\circ$ angle is a (non-regular) curvature varifold.
This simple example can be decomposed into two separate triple junctions, which are not curvature varifolds, but it is also decomposable into $3$ planes. 

\begin{figure}[H]
\captionsetup{justification=centering}
\centering
\begin{tikzpicture}[scale=0.8]
\draw (-4,0) -- (-2,0);
\draw (-3.5,-0.86) -- (-2.5,0.86);
\draw (-3.5,0.86) -- (-2.5,-0.86);	

\node [label={[label distance = -0.38 cm]0:$=$}] at (-1.5,0) {};

\draw (0,1) -- (-0.5,1.86);
\draw (0,1) -- (-0.5,0.14);
\draw (0,1) -- (1,1);
\draw (0,-1) -- (0.5,-0.14);
\draw (0,-1) -- (0.5,-1.86);
\draw (0,-1) -- (-1,-1);

\node [label={[label distance = -0.38 cm]0:$+$}] at (0,0) {};

\node [label={[label distance = -0.38 cm]0:$=$}] at (1.5,0) {};

\draw (2,0) -- (4,0);
\draw (2.5,-2.36) -- (3.5,-0.64);
\draw (2.5,2.36) -- (3.5,0.64);

\node [label={[label distance = -0.38 cm]0:$+$}] at (2.5,0.5) {};
\node [label={[label distance = -0.38 cm]0:$+$}] at (2.5,-0.5) {};

\end{tikzpicture}
\caption*{A curvature varifold with two possible decompositions: triple junction components and curvature varifold components.}
\label{fig:test4}
\end{figure}
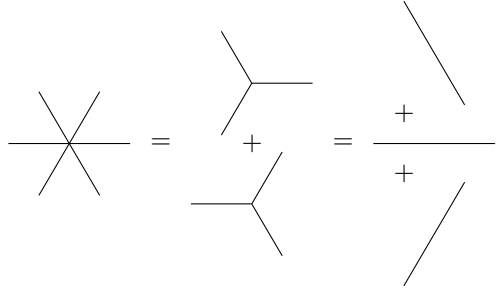

In \cite{menne-scharrer2022:arxiv}*{Example 5.10} Menne-Scharrer construct a varifold and a weakly differentiable function for which there is no decomposition that preserves weak differentiability.
Since the notion of curvature varifold is directly related to the weak differentiability of the tangent map function, which is intrinsically given by the varifold, one might ask if it is always possible to decompose a varifold such that all components preserve the differentiability property.
We will answer it in the negative by constructing a decomposable curvature varifold for which every possible component is not a curvature varifold.
The curvature varifold we obtain has a unique decomposition and its components are curvature varifolds with boundary in the sense of Mantegazza \cite{mantegazza1996}.
\bluecomment{It would be interesting to know if there exists an example so that its unique decomposition is not even a curvature varifold with boundary.}

The article is divided as follows.
In Section $2$ we compile the necessary definitions to describe decompositions of a varifold and the notion of curvature varifolds with boundary introduced by Mantegazza \cite{mantegazza1996}.
In Section $3$ we give a full description of the example and prove all its desired properties.

\textbf{Acknowledgments:} We would like to thank professor Ulrich Menne for several relevant discussions.
The author was funded by NSTC grant 113-2115-M-003-001.

\section{Preliminaries}
\noindent\textbf{Notation.} 
Let $n\in\pinteger$ be a positive integer, we denote $\{e_i\}_{i=1,\ldots,n}$ the canonical basis of $\real^n$, $\openball{x}{r}=\{y\in\real^n:|y-x|<r\}$ and $\closedball{x}{r}$ its closure.
Whenever $U\subset\real^n$ is an open set, $m,n\in\pinteger$ with $m\leq n$ we denote by \bluecomment{$\totalgrassmannian{m}{U}=U\times\grassmannian{m}{n}$ the Grassmannian over $U$}, $\varifolds_m(U)$, $\rvarifolds_m(U)$ and $\ivarifolds_m(U)$ the set of varifolds, rectifiable varifolds and integral varifolds on $U$ respectively.
Given $V\in\varifolds_m(U)$ \bluecomment{we define $\|V\|$ and $\delta V$ as in \cite{allard1972}} and for $A\subset U$ we denote $(V\restrict A)(B)=V(B\cap A\times\grassmannian{n}{m})$.
When $R\subset U$ is a $\hausdorff^m$-rectifiable set 
we write $\graphv{R}$ for the corresponding induced rectifiable $m$-varifold with density $1$ and $\tangent{R}:U\rightarrow\totalgrassmannian{m}{U}$ as $\tangent{R}(x)=(x,\tangentmap{m}{R}{x})$, which is well defined $(\hausdorff^m\restrict R)$-almost everywhere.
Whenever $P\in\grassmannian{n}{m}$ we write $P_\sharp\in\homomorphism{\real^n}{\real^n}\subset\real^{n^2}$ for the corresponding projection map in $\real^n$.
We denote $\smoothcompact{U}{\real^m}$ and $\compactfunctions{U}{\real^m}$ the spaces of $\real^m$-valued smooth compactly supported functions and continuous compactly supported functions in $U$ respectively (see \cite{menne2016.1}*{Definition 2.13} for the corresponding topologies).
The dual space $\mathcal{D}'(U,\real^m)$ denotes the space of distributions of type $\real^m$ in $U$.


\hfill

\noindent\textbf{Decomposition of Varifolds.}
The notion of a decomposition of varifolds is introduced in \cite{menne2016.1}*{Section 6} and we include it here for completion.
\begin{definition}[\cite{menne2016.1}*{5.1}]
Let $m,n\in\pinteger$, $U\subset\real^{n}$ be an open set, $V\in\varifolds_m(U)$ with $\mass{\delta V}$ a Radon measure and $E\subset U$ be a $\mass{V}+\mass{\delta V}$-measurable set.
The distributional boundary of $E$ with respect to $V$ is given by
\begin{equation*}
V\partial E = (\delta V)\restrict E - \delta(V\restrict E).
\end{equation*}
\end{definition}

\begin{definition}[\cite{menne2016.1}*{6.2}]
Let $m,n\in\pinteger$, $U\subset\real^{n}$ be an open set, $V\in\varifolds_m(U)$ with $\mass{\delta V}$ a Radon measure.
The varifold $V$ is said to be indecomposable if there exists no $\mass{V}+\mass{\delta V}$-measurable set $E\subset U$ satisfying $\mass{V}(E)>0$, $\mass{V}(U\setminus E)>0$ and $V\partial E = 0$.
\end{definition}

\begin{definition}[\cite{menne2016.1}*{6.6}]
Let $m,n\in\pinteger$, $U\subset\real^{n}$ be an open set, $V\in\varifolds_m(U)$ with $\mass{\delta V}$ a Radon measure.
A varifold $W\in\varifolds_m(U)$ is called a component of $V$ if $W\neq 0$, $W$ is indecomposable and there exists a $\mass{V}+\mass{\delta V}$-measurable set $E\subset U$ with $V\partial E=0$ such that $W= V\restrict E$.
\end{definition}

\begin{definition}[\cite{menne2016.1}*{6.9}]
Let $m,n\in\pinteger$, $U\subset\real^{n}$ be an open set, $V\in\varifolds_m(U)$ with $\mass{\delta V}$ a Radon measure.
A collection of varifolds $\Xi\subset\varifolds_m(U)$ is called a decomposition of $V$ if
\begin{enumerate}[(i)]
\item Every element of $\Xi$ is a component of $V$;
\item $V(f)=\sum_{W\in\Xi}W(f)$ for every $f\in\compactfunctions{\totalgrassmannian{m}{U}}{\real}$ and
\item $\mass{\delta V}(g)=\sum_{W\in\Xi}\mass{\delta W}(g)$ for every $g\in\compactfunctions{U}{\real}$.
\end{enumerate}
\end{definition}

We refer to \cite{menne2016.1} for further discussions and consequences of the above definitions.

\hfill

\noindent\textbf{Curvature Varifolds with Boundary Represented by Functions.}

If $\varphi\in\conecompact{\totalgrassmannian{m}{U}}$, then we write $D$ and $D^*$ for the derivative of $\varphi(x,P)$ with respect to $x$ and $P$ respectively.

The following definition was introduced in \cite{mantegazza1996}.

\begin{definition}[\cite{mantegazza1996}*{3.1}]
Let $m,n\in\pinteger$ be positive integers with $m< n$ and $U\subset\real^{n}$ be an open set.
We say that $V\in\varifolds_m(U)$ is a curvature varifold with boundary if there exists $A\in\lploc{1}(\totalgrassmannian{m}{U},\real^{n^3};V)$ and a $\real^n$-valued Radon vector measure $\partial V$ on $\totalgrassmannian{m}{U}$ such that
\begin{equation*}
\begin{aligned}
& \int_{\totalgrassmannian{m}{U}} P_{ij}D_j\varphi(x,P)+D^*_{jk}\varphi(x,P)A_{ijk}(x,P)+\varphi(x,P)A_{jij}(x,P)\integrald V(x,P) \\
& =-\int_{\totalgrassmannian{m}{U}}\varphi(x,P)\integrald {\partial_i V}(x,P)
\end{aligned}
\end{equation*}
for all $\varphi\in\conecompact{\totalgrassmannian{m}{U}}$ and $i=1,\ldots,n$.
In the above we sum over repeated indices and $\partial_i V$ is the signed measure $\partial_i V=\innerproduct{\partial V}{e_i}$.
We denote by $\cbvarifolds_m(U)$ the space of curvature $m$-varifolds with boundary in $U$.
Whenever needed we will simplify notation and write $\boundary_i(V,\varphi)$ to denote the left-hand side of the above definition and $\boundary(V,\varphi)=(\boundary_i)_{i=1,\ldots,n}$.
In particular, $V$ is a curvature varifold (without boundary) in the sense of Hutchinson \cite{hutchinson1986}*{5.2.1} if $\partial V=0$ and its space is denoted by $\cvarifolds_m(U)$.
\end{definition}

\begin{remark}
We say that a $\real^n$-valued measure is Radon if each coordinate signed measure is Radon.
\end{remark}

Let $n\in\pinteger$, $\Omega\subset\real^{n-1}$ be an open set, $g:\Omega\rightarrow\real$ be a function of class $C^2$.
Denote by $\Sigma=\{(y,g(y))\in\Omega\times\real:y\in\Omega\}$ the graph of $g$ and $\graphv{\Sigma}\in\ivarifolds_{n-1}(\Omega\times\real)$ the integral varifold corresponding to the graph of $g$ with density $1$.
We write $T(x)=\tangentmap{n-1}{\Sigma}{x}_\sharp\in\real^{n^2}$ for $x\in\Sigma$ and compute
\bluecomment{
\begin{equation*}
\begin{aligned}
& T_{ij}=\delta_{ij}- \partial_i g\partial_j g(1+|\nabla g|^2)^{-1}, i,j=1,\ldots,n-1;\\
& T_{in}=T_{ni}={\partial_ig}(1+|\nabla g|^2)^{-1}, i=1,\ldots,n-1;\\
& T_{nn}=1-(1+|\nabla g|^2)^{-1}.
\end{aligned}
\end{equation*}
}
We note that if $\Omega$ has regular boundary then $\graphv{\Sigma}$ is a curvature varifold with boundary with respect to $A(x)_{ijk}=T(x)_{il}\partial_l T(x)_{jk}$ (see \cite{mantegazza1996}*{p.811 (2.1)}).
\bluecomment{To simplify notation we write $A(x)_{ijk}$ instead of $A_{ijk}(x,T(x))$ for $x\in\Sigma$}.
\bluecomment{
A direct computation gives for each $l=1,\ldots,n-1$:
\begin{equation*}
\begin{aligned}
\partial_l T_{jk} & = - (\partial_l\partial_j g \partial_k g + \partial_j g\partial_l\partial_k g)(1+|\nabla g|^2)^{-1}\\
                  & \quad + 2\partial_j g\partial_k g(\sum_{m=1}^{n-1}\partial_m g\partial_l\partial_mg)(1+|\nabla g|^2)^{-2}, j,k=1,\ldots,n-1;\\
\partial_l T_{jn}=\partial_l T_{nj} & = \partial_l\partial_jg(1+|\nabla g|^2)^{-1}-2\partial_jg(\sum_{m=1}^{n-1}\partial_m g\partial_l\partial_mg)(1+|\nabla g|^2)^{-2},\\
                  & \qquad j=1,\ldots,n-1;\\
\partial_l T_{nn} & = (\sum_{m=1}^{n-1}\partial_m g\partial_l\partial_mg)(1+|\nabla g|^2)^{-2}.
\end{aligned}
\end{equation*}
}
\bluecomment{
\begin{remark}\label{Aijk remark}
We note that $|T_{ij}|\leq 1$ for all $i,j=1,\ldots, n$ and for some constant $C(n)$ depending only on $n$ we have $|\partial_l T_{jk}|\leq C(n)|D^2 g|$ for all $l=1,\ldots,n-1$ and $j,k=1,\ldots,n$.
It follows that $|A_{ijk}|\leq C(n)|D^2 g|$ for all $i,j,k=1,\ldots,n$ and some other constant constant $C(n)$ depending only on $n$.
\end{remark}
}
\begin{lemma}\label{graph varifolds}
If $\Omega$ has $C^1$ boundary and $g$ is $C^1$ along $\partial \Omega$ with respect to the inward conormal, then the above choice of $A_{ijk}$ makes $\graphv{\Sigma}$ into a curvature varifold with boundary $\partial\graphv{\Sigma}=\tangent{\Sigma}_\sharp(\nu\hausdorff^{n-1}\restrict\partial\Sigma)$, where $\nu$ is the inward conormal of $\partial\Sigma$.
\end{lemma}
\begin{proof}
It follows directly from the divergence theorem.
\end{proof}

\section{Main Example}

First we take $\Phi:\real\rightarrow\real$ a bump function with the following properties:
\begin{enumerate}[(a)]
\item $0\leq\Phi\leq 1$, $\Phi(0)=1$;
\item $\support\Phi\subset(-1,1)$;
\item $\Phi(-t)=\Phi(t)$ for all $t\geq 0$;
\item $\Phi'(t)\leq 0$ for all $t\geq 0$;
\item $\sup|\Phi'|\leq 4$ and
\item $\int_\real\Phi(t)\integrald\lebesgue{1}t=1$.
\end{enumerate}

Let us define $\Omega=\{x\in\real^2:x_1>0\}$ and a function $g:\Omega\rightarrow\real$ as
\begin{equation*}
\begin{aligned}
g(x_1,x_2) & = \frac{1}{\sqrt{3}}\int_{-x_2}^{x_2}\Phi\left(\frac{t}{x_1}\right)\integrald\lebesgue{1}t\\
					 & = \frac{x_1}{\sqrt{3}}\int^{\frac{x_2}{x_1}}_{-\frac{x_2}{x_1}}\Phi\left(\tau\right)\integrald\lebesgue{1}\tau.
\end{aligned}
\end{equation*}

\begin{lemma}\label{function g}
The function $g$ defined above is smooth and satisfies:
\begin{enumerate}[(i)]
\item $g(x_1,0)=0$;
\item $g(x_1,x_2)=\frac{x_1}{\sqrt{3}}\frac{x_2}{|x_2|}$ on $\{x\in\Omega:|x_2|>x_1\}$ and
\item $\int_{K\cap\Omega}|D^2g|\integrald\lebesgue{2}<\infty$ for all compact sets $K\subset\real^2$.
\end{enumerate}
\end{lemma}
\begin{proof}
First we note that properties (i) and (ii) follow from trivial calculations.
Next we compute the first partial derivatives of $g$
\begin{equation*}
\begin{aligned}
 \partial_1g & = \frac{1}{\sqrt{3}}\int_{-\frac{x_2}{x_1}}^{\frac{x_2}{x_1}}\Phi(\tau)\integrald\lebesgue{1}\tau-\frac{2x_2}{x_1\sqrt{3}}\Phi\left(\frac{x_2}{x_1}\right),\\
 \partial_2g & = \frac{2}{\sqrt{3}}\Phi\left(\frac{x_2}{x_1}\right).
\end{aligned}
\end{equation*}
It follows that $|\partial_1 g|\leq\frac{3}{\sqrt{3}}$ and $|\partial_2 g|\leq \frac{2}{\sqrt{3}}$ whenever $|x_2|<x_1$.

We compute the second partial derivatives:
\begin{equation*}
\begin{aligned}
\partial_1^2g & = \frac{2x_2^2}{x_1^3\sqrt{3}}\Phi'\left(\frac{x_2}{x_1}\right),\\
\partial_1\partial_2g=\partial_2\partial_1g & = \frac{-2x_2}{x_1^2\sqrt{3}}\Phi'\left(\frac{x_2}{x_1}\right) \text{ and }\\
\partial_2^2g & = \frac{2}{x_1\sqrt{3}}\Phi'\left(\frac{x_2}{x_1}\right).
\end{aligned}
\end{equation*}
Hence, $D^2g(x)=0$ on $\{x\in\Omega:x_1<|x_2|\}$ so we conclude by using polar coordinates that
\begin{equation*}
\int_{(0,1]\times[-1,1]}|D^2g|(x)\integrald\lebesgue{2}x<\infty,
\end{equation*}
which is sufficient to prove the final statement.
\end{proof}

\begin{remark}\label{function g remark}
We observe that $g$ can be extended smoothly to $\bar\Omega\setminus\{(0,0)\}$.
If we define $A_{ijk}(x)$ for $x\in\real^3$ and $i,j,k=1,2,3$ with respect to the above function $g$ as in Lemma \ref{graph varifolds}, then we have $A_{ijk}(x)=0$ on $\{x\in\real^3:0<x_1\leq |x_2|,x_3=g(x_1,x_2)\}$ and $|A_{ijk}(x)|\leq C|D^2g(x)|$ on $\{x\in\real^3:|x_2|< x_1,x_3=g(x_1,x_2)\}$ for all $i,j,k=1,2,3$ and some positive constant $C>0$ (see Remark \ref{Aijk remark}).
Therefore $\int_{K\cap\graph{g}} |A_{ijk}|(x)\integrald\hausdorff^2 x<\infty$ for every compact set $K\subset\real^3$ and $i,j,k=1,2,3$.
Furthermore, from the same calculations as above it follows that $\int_{\openball{0}{\varepsilon}\cap\graph{g}}|A_{ijk}|\integrald\hausdorff^2 x$ tends to $0$ as $\varepsilon$ tends to $0$.
\end{remark}

Let us denote $\Sigma_1=\{(x,g(x))\in\real^3:x\in\Omega\}$, $\Omega^\pm=\{x\in\Omega:\pm x_2>0\}$, $L=\{x\in\real^3:x_1=x_3=0\}$ and $L^\pm=\{x\in L:\pm x_2>0\}$.
We define $\Sigma_1^\pm=\{(x,g(x))\in\real^3:x\in\Omega^\pm\}$ and observe that the inward conormal vector field of $\Sigma_1^\pm$ along $L^\pm$ is given by $\nu_1^\pm=(\frac{\sqrt{3}}{2},0,\pm\frac{1}{2})$.
Let $\rot,\reflect:\real^3\rightarrow\real^3$ denote the rotation by $\frac{2\pi}{3}$ around the $x_2$-axis and the reflection across the $x_1x_3$-plane respectively.

Next we define:
\begin{equation*}
\begin{aligned}
& \Sigma_2 = \reflect(\Sigma_1),\\
& \Sigma_3 = \rot(\Sigma_1),\\
& \Sigma_5 = \rot(\Sigma_3),\\
& \Sigma_4 = \rot(\Sigma_2),\\
& \Sigma_6 = \rot(\Sigma_4)\\
\end{aligned}
\end{equation*}
and similarly $\Sigma_i^\pm$ for $i=2,\ldots,6$.
We also define $\nu_2^\pm=(\frac{\sqrt{3}}{2},0,\mp\frac{1}{2})$ and $\nu_i^\pm=\rot(\nu_{i-2}^\pm)$ for $i=3,\ldots,6$.
Note that $\nu_i^\pm$ is the inward conormal of $\Sigma_i^\pm$ along the boundary component $L^\pm$.

Finally, we write $W_i=\graphv{\Sigma_i}$ for $i=1,\ldots,6$ and the main example is given by $V=\sum_{i=1}^6W_i$.
We further write $Z_1=W_1+W_3+W_5$ and $Z_2=W_2+W_4+W_6$.
We will later prove in Remark \ref{z indecomposable} that $Z_1$ and $Z_2$ are indecomposable and $\Xi=\{Z_1,Z_2\}$ is a decomposition of $V$.
In the remainder of the section we will prove that $V$ is a curvature varifold (without boundary), $\Xi=\{Z_1,Z_2\}$ is the unique decomposition of $V$ and each $Z_1,Z_2$ are not curvature varifolds (without boundary).

Denote $t_1=e_1$, $t_2=\rot(t_1)$, $t_3=\rot(t_2)$, $T_k=\{\lambda t_k:\lambda> 0\}$ for $k=1,2,3$ and $T=\cup_{k=1}^3T_k$.
Let $C_k=\{x\in\real^3\setminus\{0\}:\innerproduct{\frac{x}{|x|}}{t_k}>\sqrt{\frac{3}{7}}\}$ be the open half-cone with central axis $T_k$ and angle $\cos^{-1}(\sqrt{\frac{3}{7}})$ for $k=1,2,3$, which is the angle between the half-line $\{(x_1,x_1,g(x_1,\pm x_1)):x_1>0\}$ and the $x_1$-axis.
Further write $C=\cup_{k=1}^3C_k$, $D=\real^3\setminus\bar{C}$ and $D^\pm=\{x\in D:\pm x_2>0\}$.
We also define $\eta_1^\pm=\pm\sqrt{\frac{3}{7}}(0,1,\frac{2}{\sqrt{3}})$, $\eta_2^\pm=\pm\sqrt{\frac{3}{7}}(0,1,-\frac{2}{\sqrt{3}})$ and $\eta_i^\pm=\rot(\eta_{i-2}^\pm)$ for $i=3,\ldots,6$.
Note that $\eta_i^\pm$ is the inward conormal of $\Sigma_i^\pm$ along the boundary component $T_{\lceil \frac{i}{2} \rceil}=(\bar{\Sigma}_i^\pm\setminus\Sigma_i^\pm)\cap C_{\lceil \frac{i}{2} \rceil}$. 

\setlength{\intextsep}{10pt}
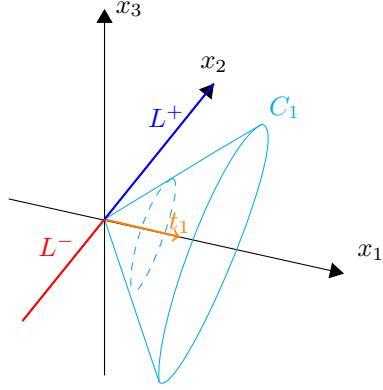
\begin{figure}[H]
\captionsetup{justification=centering}
\centering
  \begin{tikzpicture}[scale=0.9]

\draw[-triangle 60] (0,-2.3) -- (0,3.1);
\draw[-triangle 60] (-1.4,0.3) -- (3.5,-0.8);
\node [label={[label distance = -0.1 cm]45:$x_1$}] at (3.5,-0.8) {};
\node [label={[label distance = -0.1 cm]0:$x_3$}] at (0,3.1) {};
\draw [-triangle 60](-1.2,-1.5) -- (1.6,2);
\node [label={[label distance = -0.1 cm]90:$x_2$}] at (1.6,2) {};

\draw[cyan] (0,0) -- (0.25,-0.75) -- (0.55,-1.65) -- (0.8,-2.4);
\draw[cyan] (0,0) -- (0.65,0.4) -- (1.625,1) -- (2.3,1.4);
\draw[cyan] (2.3,1.4) .. controls (1.7,1.1) and (0.7,-1.6) .. (0.8,-2.4);
\draw[cyan] (2.3,1.4) .. controls (2.8,1.3) and (1.1,-2.7) .. (0.8,-2.4);
\draw[dashed,cyan] (1,0.6) .. controls (0.7,0.5) and (0.3,-0.6) .. (0.4,-1.1);
\draw[dashed,cyan] (1,0.6) .. controls (1.2,0.5) and (0.6,-1.1) .. (0.4,-1.1);

\draw[->,orange,thick] (0,0)  -- (1.1,-0.26);
\node [label={[label distance = -0.2 cm]90:{\color{orange}$t_1$}}] at (1.1,-0.26) {};
\node [label={[label distance = -0.2 cm]60:{\color{cyan}$C_1$}}] at (2.3,1.4) {};

\draw[blue,thick] (0,0) -- (1.6,2);
\node [label={[label distance = -0.1 cm]90:{\color{blue}$L^+$}}]  at (0.9,1.2) {};
\draw[red,thick] (0,0) -- (-1.2,-1.5);
\node [label={[label distance = -0.1 cm]90:{\color{red}$L^-$}}]  at (-0.7,-0.7) {};
\end{tikzpicture}
  \caption*{Non-planar region $C_1$ and boundary of planar region $L^+,L^-$.}
\label{fig:cone}
\end{figure}

We note that $\Sigma_{2k+1}\cap\Sigma_{2k}=T_k$ for $k=1,2,3$ and the intersection is transversal, which can be observe for $\Sigma_1$ and $\Sigma_2$ since $\partial_2g(x_1,0)>0$, as it can be seen in the proof of Lemma \ref{function g}, the tangent vector $(x_1,0,\partial_2g(x_1,0))$ on $\Sigma_1$ along $T_1$ is reflected to $(x_1,0,-\partial_2g(x_1,0))$ on $\Sigma_2$ along $T_1$, whilst the tangent direction $(1,0,0)$ along $T_1$ is preserved.
The transversality if $\Sigma_{2k+1}\cap\Sigma_{2k}$ for $k=2,3$ is then preserved by rotation.
Define the planes $P_i^\pm=\{\lambda\nu_i^\pm+\mu e_2:\lambda,\mu\in\real\}$ for $i=1,\ldots,6$.
Due to the choice of angle for the cones $C_k$ we have precisely that $\Sigma_i^{\pm}\cap D=P_i^{\pm}\cap \{x\in D^\pm:\innerproduct{x}{\nu_i^\pm}>0\}$ for $i=1,\ldots,6$.
Therefore, $\tangentmap{2}{W_i}{x}=P_i^\pm$ for all $x\in\support{\mass{W_i}}\cap D^\pm$ and $i=1,\ldots,6$.

\begin{figure}[H]
\centering
\begin{subfigure}{.5\textwidth}
  \centering
  \begin{tikzpicture}[scale=0.3]

\draw[-triangle 60] (0,-5) -- (0,5);
\draw[-triangle 60] (-5,0) -- (5,0);
\node [label={[label distance = -0.1 cm]45:$x_1$}] at (5,0) {};
\node [label={[label distance = -0.1 cm]60:$x_3$}] at (0,5) {};

\begin{scope}[scale=0.8]
\draw [->,blue,thick](0,0) -- (4,2.3);
\node [label={[label distance = -0.1 cm]0:{\color{blue}$\nu_1^+$}}] at (4,2.3) {};
\draw [->,blue,thick](0,0) -- (-4,2.3);
\node [label={[label distance = -0.1 cm]90:{\color{blue}$\nu_3^+$}}] at (-4,2.3) {};
\draw [->,blue,thick](0,0) -- (0,-4.6);
\node [label={[label distance = -0.1 cm]0:{\color{blue}$\nu_5^+$}}] at (0,-4.6) {};

\draw [->,red,thick](0,0) -- (0,4.6);
\node [label={[label distance = -0.1 cm]0:{\color{red}$\nu_4^+$}}] at (0,4.6) {};
\draw [->,red,thick](0,0) -- (-4,-2.3);
\node [label={[label distance = -0.1 cm]90:{\color{red}$\nu_6^+$}}] at  (-4,-2.3) {};
\draw [->,red,thick](0,0) -- (4,-2.3);
\node [label={[label distance = -0.1 cm]0:{\color{red}$\nu_2^+$}}] at (4,-2.3) {};
\end{scope}
\end{tikzpicture}
  \caption*{Conormal on planar region at $x_2>0$.}
  \label{fig:nuplus}
\end{subfigure}%
\begin{subfigure}{.5\textwidth}
  \centering
  \begin{tikzpicture}[scale=0.3]

\draw[-triangle 60] (0,-5) -- (0,5);
\draw[-triangle 60] (-5,0) -- (5,0);
\node [label={[label distance = -0.1 cm]45:$x_1$}] at (5,0) {};
\node [label={[label distance = -0.1 cm]60:$x_3$}] at (0,5) {};

\begin{scope}[scale=0.8]
\draw [->,red,thick](0,0) -- (4,2.3);
\node [label={[label distance = -0.1 cm]0:{\color{red}$\nu_2^-$}}] at (4,2.3) {};
\draw [->,red,thick](0,0) -- (-4,2.3);
\node [label={[label distance = -0.1 cm]90:{\color{red}$\nu_4^-$}}] at (-4,2.3) {};
\draw [->,red,thick](0,0) -- (0,-4.6);
\node [label={[label distance = -0.1 cm]0:{\color{red}$\nu_6^-$}}] at (0,-4.6) {};

\draw [->,blue,thick](0,0) -- (0,4.6);
\node [label={[label distance = -0.1 cm]0:{\color{blue}$\nu_3^-$}}] at (0,4.6) {};
\draw [->,blue,thick](0,0) -- (-4,-2.3);
\node [label={[label distance = -0.1 cm]90:{\color{blue}$\nu_5^-$}}] at  (-4,-2.3) {};
\draw [->,blue,thick](0,0) -- (4,-2.3);
\node [label={[label distance = -0.1 cm]0:{\color{blue}$\nu_1^-$}}] at (4,-2.3) {};
\end{scope}
\end{tikzpicture}
  \caption*{Conormal on planar region at $x_2<0$.}
  \label{fig:numinus}
\end{subfigure}
\label{fig:test}
\end{figure}

\begin{remark}\label{vector identities}
With the above notation and definitions we have the following identities:
\begin{equation*}
\begin{aligned}
& \nu_1^+ + \nu_6^+ = 0, P_1^+=P_6^+\\
& \nu_3^+ + \nu_2^+ = 0, P_3^+=P_2^+\\
& \nu_5^+ + \nu_4^+ = 0, P_5^+=P_4^+\\
& \nu_1^- + \nu_4^- = 0, P_1^-=P_4^-\\
& \nu_3^- + \nu_6^- = 0, P_3^-=P_6^-\\
& \nu_5^- + \nu_2^- = 0, P_5^-=P_2^- \text{ and }\\
\nu_1^\pm+&\nu_3^\pm+\nu_5^\pm=\nu_2^\pm+\nu_4^\pm+\nu_6^\pm=0.
\end{aligned}
\end{equation*}
\end{remark}

\begin{lemma}\label{w_i boundary}
$W_i\in\cbvarifolds_2(\real^3)$ is a curvature varifold with boundary given by
\begin{equation*}
\partial W_i=\tangent{W_i}_\sharp(\nu_i^+\hausdorff^1\restrict L^+ + \nu_i^-\hausdorff^1\restrict L^-)
\end{equation*}
 for each $i=1,\ldots,6$.
In particular $Z_1$ and $Z_2$ are curvature varifolds with non-zero boundary.
\end{lemma}
\begin{proof}
Let $\varphi\in\conecompact{\totalgrassmannian{2}{\real^3}}$ be an arbitrary function and $\varepsilon>0$.
Take $\psi_\varepsilon:\real^3\rightarrow\real$ a smooth cut-off function satisfying:
\begin{enumerate}[(a)]
\item $0\leq\psi_\varepsilon(x)\leq 1$,
\item $\psi_\varepsilon(x)=1$ for all $x\in\openball{0}{\varepsilon}$,
\item $\support{\psi_\varepsilon}\subset\openball{0}{2\varepsilon}$ and
\item $|\nabla\psi_\varepsilon|\leq\frac{2}{\varepsilon}$.
\end{enumerate}
Define $\varphi_\varepsilon(x,P)=\psi_\varepsilon(x)\varphi(x,P)$ and $\bar\varphi_\varepsilon(x,P)=(1-\psi_\varepsilon(x))\varphi(x,P)$ so that $\varphi=\varphi_\varepsilon+\bar\varphi_\varepsilon$, $\support{\varphi_\varepsilon}\subset\openball{0}{2\varepsilon}\times\grassmannian{3}{2}$ and $\support{\bar\varphi_\varepsilon}\subset(\real^3\setminus\closedball{0}{\varepsilon})\times\grassmannian{3}{2}$.
Observe that for each $l=1,2,3$ we have $\boundary_l(W_i,\varphi)=\boundary_l(W_i,\varphi_\varepsilon)+\boundary_l(W_i,\bar\varphi_\varepsilon)$.

Let $A_{ljk}(x)$ be defined by $g$ as in Lemma \ref{graph varifolds} and, as in Section $2$, we simplify notation by writing $A_{ljk}(x,P)=A_{ljk}(x)$ when $P=T_x\Sigma_i$ and $0$ otherwise.
It follows from Lemma \ref{function g} that $A\in\lploc{1}(\totalgrassmannian{2}{\real^3},\real^{3^3};W_i)$.
We compute for each $l=1,2,3$:
\begin{equation*}
\begin{aligned}
\boundary_l(W_i,\varphi_\varepsilon) = & \int_{{\real^3}}(T_x\Sigma_i)_{lj}D_j\varphi_\varepsilon(x,T_x\Sigma_i)\\
                                       & \quad +D_{jk}^*\varphi_\varepsilon(x,T_x\Sigma_i)A_{ljk}(x)+\varphi_\varepsilon(x,T_x\Sigma_i)A_{jlj}(x)\integrald\mass{W_i}x\\
																		 = & \int_{\openball{0}{2\varepsilon}}(T_x\Sigma_i)_{lj}(D_j\psi_\varepsilon(x)\varphi(x,T_x\Sigma_i)+\psi_\varepsilon(x)D_j\varphi(x,T_x\Sigma_i))\\
                                       & \quad +\psi_\varepsilon(x)(D_{jk}^*\varphi(x,T_x\Sigma_i)A_{ljk}(x)+\varphi(x,T_x\Sigma_i)A_{jlj}(x))\integrald\mass{W_i}x\\
																  \leq & (\frac{2n}{\varepsilon}\sup|\varphi| + n\sup|\nabla\varphi|)\mass{W_i}(\openball{0}{2\varepsilon})\\
																	     & \quad + (n^2\sup|\nabla^*\varphi|+n\sup|\varphi|)\sup_{jk}\int_{\openball{0}{2\varepsilon}}|A_{ijk}(x)|\integrald\mass{W_i}x.
\end{aligned}
\end{equation*}
Hence, $\lim_{\varepsilon\rightarrow 0}\boundary_l(W_i,\varphi_\varepsilon)=0$, where the first term tends to $0$ since $|Dg|^2$ is uniformly bounded on $\openball{0}
{1}\cap\{x\in\real^3:x\in\graph{g}\}$ and the second term tends to $0$ from Remark \ref{function g remark}.
Similarly we compute
\begin{equation*}
\begin{aligned}
\boundary_l(W_i,\bar\varphi_\varepsilon) = & \int_{\real^3\setminus\closedball{0}{\varepsilon}}(T_x\Sigma_i)_{lj}D_j\bar\varphi_\varepsilon(x,T_x\Sigma_i)\\
                               & \quad +D_{jk}^*\bar\varphi_\varepsilon(x,T_x\Sigma_i)A_{ljk}(x)+\bar\varphi_\varepsilon(x,T_x\Sigma_i)A_{jlj}(x)\integrald\mass{W_i}x.\\
\end{aligned}
\end{equation*}
It follows from the Divergence Theorem on $\Sigma_i$ with respect to the vectorfield $\bar{\varphi}_\varepsilon(x,T_x\Sigma_i)(T_x\Sigma_i)_\sharp e_l$ that
\begin{equation*}
\begin{aligned}
\boundary_l(W_i,\bar\varphi_\varepsilon) = & -\int_{L\cap(\real^3\setminus\closedball{0}{\varepsilon})}\bar\varphi_\varepsilon(x,T_x\Sigma_i)\innerproduct{\nu(\Sigma_i)}{e_l}\integrald\hausdorff^1x\\
                                 = & -\int_{L^+\setminus\closedball{0}{\varepsilon}}\bar\varphi_\varepsilon(x,T_x\Sigma_i)\innerproduct{\nu_i^+}{e_l}\integrald\hausdorff^1x\\
																   & \quad -\int_{L^-\setminus\closedball{0}{\varepsilon}}\bar\varphi_\varepsilon(x,T_x\Sigma_i)\innerproduct{\nu_i^-}{e_l}\integrald\hausdorff^1x.
\end{aligned}
\end{equation*}
By letting $\varepsilon$ tend to zero we have $\boundary_l(W_i,\varphi)=\lim_{\varepsilon\rightarrow 0}\boundary_l(W_i,\bar\varphi_\varepsilon)$, that is,
\begin{equation*}
\begin{aligned}
\boundary_l(W_i,\varphi) = & -\int_{L\cap(\real^3\setminus\{0\})}\varphi(x,T_x\Sigma_i)(\innerproduct{\nu_i^+}{e_l}+\innerproduct{\nu_i^-}{e_l})\integrald\hausdorff^1x\\
        								 = & -\int_{\totalgrassmannian{2}{\real^3}}\varphi(x,P)\partial_lW_i,
\end{aligned}
\end{equation*}
where $\partial_lW_i=\tangent{W_i}_\sharp(\innerproduct{\nu_i^+}{e_l}\hausdorff^1\restrict L^+ + \innerproduct{\nu_i^-}{e_l}\hausdorff^1\restrict L^-)$.
\end{proof}

\begin{remark}\label{remark not curvature}
Denote by $\pi:\totalgrassmannian{2}{D}\rightarrow\grassmannian{2}{3}$ the projection onto the Grassmannian space.
It follows that $\pi_\sharp\partial W_i=\nu_i^+\delta_{P_i^+}+\nu_i^-\delta_{P_i^-}$, where $\delta_{P_i^\pm}$ is the Dirac measure centered at $P_i^\pm$.
Therefore $\pi_\sharp\partial Z_k=\sum_{j=1}^3\nu_{2j-k}^+\delta_{P_{2j-k}^+}+\nu_{2j-k}^-\delta_{P_{2j-k}^-}\neq 0$ for $k=1,2$.
That is, $Z_1,Z_2$ are not curvature varifolds (without boundary).
\end{remark}

\begin{figure}[H]
\centering
  \begin{tikzpicture}[scale=0.8]

\draw[-triangle 60] (0,0) node (v1) {} -- (0,3.5);
\draw[-triangle 60] (-3.5,1) -- (3.5,-1);
\node [label={[label distance = -0.1 cm]45:$x_1$}] at (3.5,-1) {};
\node [label={[label distance = -0.1 cm]0:$x_3$}] at (0,3.5) {};
\draw [-triangle 60](-3.5,-4) -- (3,3.5);
\node [label={[label distance = -0.1 cm]90:$x_2$}] at (3,3.5) {};

\draw[blue,thick] (2.58,3) -- (5,4);
\draw[blue,thick] (0.85,1) -- (3.25,2);
\draw[blue,thick] (1.35,1.6) -- (3.78,2.6);
\draw[blue,thick] (1.95,2.3) -- (4.38,3.3);
\draw[blue,thick] (5,4) -- (3.25,2);
\draw[blue,thick] (0.4,0.5) -- (1.625,1);
\draw[blue,thick] (0.15,0.2) -- (0.65,0.4);
\draw[blue,thick] (3.62,3.42) -- (1.625,1);
\draw[blue,thick] (3.12,3.22) -- (0.65,0.4);

\draw[red,thick] (-0.9,-1) -- (1,-3);
\draw[red,thick] (-2.65,-3) -- (-0.75,-5);
\draw[red,thick] (1,-3) -- (-0.75,-5);
\draw[red,thick] (-2.02,-2.3) -- (-0.1,-4.27);
\draw[red,thick] (-1.42,-1.6) -- (0.48,-3.6);
\draw[red,thick] (-0.55,-0.57) -- (0.55,-1.65);
\draw[red,thick] (-0.26,-0.25) -- (0.25,-0.75);
\draw[red,thick] (0.55,-1.65) -- (-1.61,-4.11);
\draw[red,thick] (0.25,-0.75) -- (-2.2,-3.5);

\draw[orange,thick] (1,-3) .. controls (1.4,-2.5) and (1.2,-1.5) .. (1.5,-0.5);
\draw[orange,thick] (1,-0.3) .. controls (0.9,-0.7) and (0.9,-1.3) .. (0.55,-1.65);
\draw[orange,thick] (0.4,-0.1) .. controls (0.38,-0.3) and (0.4,-0.52) .. (0.25,-0.75);
\draw[cyan,thick] (1.5,-0.5) .. controls (1.8,0.5) and (2.8,1.5) .. (3.25,2);
\draw[cyan,thick] (1.625,1) .. controls (1.4,0.7) and (1.1,0.1) .. (1,-0.3);
\draw[cyan,thick] (0.65,0.4) .. controls (0.48,0.18) and (0.42,0.1) .. (0.4,-0.1);

\draw[dashed] (0,0) -- (0,-4);
\draw (0,-4) -- (0,-5);
\draw[dotted] (0,0) -- (0.25,-0.75) -- (0.55,-1.65) -- (1,-3);
\draw[dotted] (0,0) -- (0.65,0.4) -- (1.625,1) -- (3.25,2);
\node [label={[label distance = -0.1 cm]90:$W_1$}] at (0.9,2) {};

\node [label={[label distance = 0.1 cm]0:{\color{blue}$\Sigma_1^+\cap D^+$}}] at (4.4,3.3) {};
\node [label={[label distance = 0.1 cm]0:{\color{cyan}$\Sigma_1^+\cap C_1$}}] at (2.1,0.5) {};
\node [label={[label distance = 0.1 cm]0:{\color{orange}$\Sigma_1^-\cap C_1$}}] at (1.4,-1.5) {};
\node [label={[label distance = 0.1 cm]0:{\color{red}$\Sigma_1^-\cap D^-$}}] at (0.5,-3.6) {};
\end{tikzpicture}
  \caption*{Building block $W_1$ and partitioning of $\Sigma_1$.}
\label{fig:w1}
\end{figure}
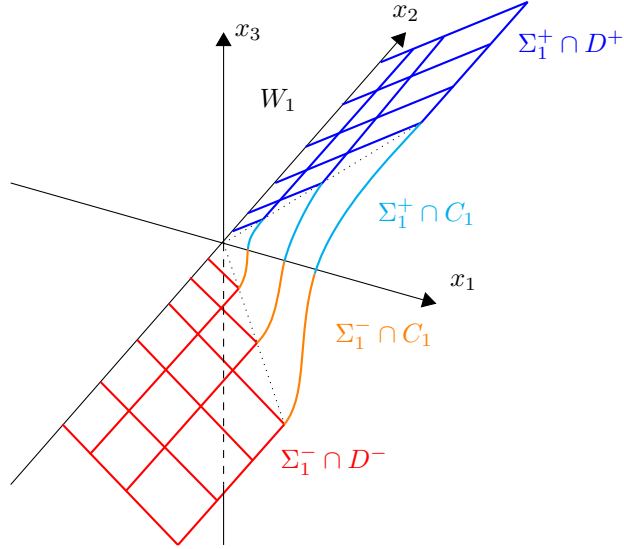

\begin{lemma}\label{distributional boundary pm}
The distributional boundary of $\Sigma_i^\pm$ with respect to $V$ is given by
\begin{equation*}
V\partial\Sigma_i^\pm(Y)=\int_{L^\pm}\innerproduct{Y(x)}{\nu_i^\pm}\integrald\hausdorff^1x+\int_{T_{\lceil \frac{i}{2} \rceil}}\innerproduct{Y(x)}{\eta_i^\pm}\integrald\hausdorff^1x,
\end{equation*}
for all $Y\in\smoothcompact{\real^3}{\real^3}$ and $i=1,\ldots,6$.
\end{lemma}
\begin{proof}
Note that Lemma \ref{function g} implies that the generalized mean curvature of $V$ is in $\lploc{1}(\real^3,\real^3;\mass{V})$ so the result follows from \cite{allard1972}*{4.7} and a similar cut-off function argument as in Lemma \ref{w_i boundary}.
\end{proof}

\begin{corollary}\label{distributional boundary total}
The distributional boundary of $\Sigma_i$ with respect to $V$ is given by
\begin{equation*}
V\partial\Sigma_i(Y)=\int_{L^+}\innerproduct{Y(x)}{\nu_i^+}\integrald\hausdorff^1x+\int_{L^-}\innerproduct{Y(x)}{\nu_i^-}\integrald\hausdorff^1x,
\end{equation*}
for all $Y\in\smoothcompact{\real^3}{\real^3}$ and $i=1,\ldots,6$.
\end{corollary}
\begin{proof}
Follows directly from the above result and $\eta_i^+=-\eta_i^-$.
\end{proof}

\begin{lemma}\label{constancy}
Let $X\in\varifolds_2(\real^3)$ be a component of $V$ and $i\in\{1,\ldots,6\}$.
If $\mass{X}(\Sigma_i^\pm)>0$, then $\Sigma_i^\pm\subset\support{\mass{X}}$.
\end{lemma}
\begin{proof}
Let $E\subset\real^{3}$ be a $\mass{V}+\mass{\delta V}$-measurable set with $V\partial E=0$ and $X=V\restrict E$.
Suppose by contradiction that $\mass{X}(\Sigma_i^+)>0$ and $\Sigma_i^+\setminus\support{\mass{X}}\neq\emptyset$.
Then there exists $x\in\Sigma_i^+$ and $\varepsilon>0$ such that $\mass{X}(\openball{x}{\varepsilon})=0$, which implies that $\mass{V}(\openball{x}{\varepsilon}\cap E)=0$.
Hence, $\mass{V}(\Sigma_i^+\setminus E)\geq\mass{V}((\Sigma_i^+\setminus E)\cap\openball{x}{\varepsilon})=\mass{V}(\Sigma_i^+\cap\openball{x}{\varepsilon})>0$.

Since $\Sigma_i^+$ is a regular surface, $\support{(V-\graphv{\Sigma_i^+)}}\subset (\real^3\setminus\Sigma_i^+)\times\grassmannian{2}{3}$ and $V\partial E=0$, then the above contradicts the Constancy Lemma \cite{menne-scharrer2022:arxiv}*{Lemma 6.1}. Similarly we obtain a contradiction in the case of $\Sigma_i^-$, which concludes the proof.
\end{proof}

\begin{lemma}\label{plus implies minus}
Let $X\in\varifolds_2(\real^3)$ be a component of $V$ and $i\in\{1,\ldots,6\}$.
Then $\mass{X}(\Sigma_i^+)>0$ if and only if $\mass{X}(\Sigma_i^-)>0$.
In particular, if either $\mass{X}(\Sigma_i^+)>0$ or $\mass{X}(\Sigma_i^-)>0$, then $\Sigma_i\subset\support{\mass{X}}$.
\end{lemma}
\begin{proof}
Let $E\subset\real^{3}$ be a $\mass{V}+\mass{\delta V}$-measurable set with $V\partial E=0$ and $X=V\restrict E$.
Suppose by contradiction that $\mass{X}(\Sigma_i^+)>0$ and $\mass{X}(\Sigma_i^-)=0$.
It follows from Lemma \ref{constancy} that $\Sigma_i^+\subset\support{\mass{X}}\subset E$, hence $\support{\mass{X}}\cap C_{\lceil\frac{i}{2}\rceil}\neq\emptyset$.

Let $J=\{j\in\{1,\ldots,6\}:\mass{X}(\Sigma_j\cap C_{\lceil\frac{i}{2}\rceil})>0\}$.
Observe that $\support{\mass{V}}\cap C_{k}=(\Sigma_{2k-1}\cup\Sigma_{2k})\cap C_k$ for $k=1,2,3$, that is, $\cardinality{J}\leq 2$.
In particular, if $i$ is even then either $J=\{i\}$ or $J=\{i,i-1\}$ and if $i$ is odd then either $J=\{i\}$ or $J=\{i,i+1\}$.

Without loss of generality we may assume that $i$ is odd.
If $J=\{i\}$ then $\mass{X}((E\setminus\Sigma_i^+)\cap C_{\lceil\frac{i}{2}\rceil})=0$ by the contradiction assumption and from Lemma \ref{distributional boundary pm} we have
\begin{equation*}
V\partial E (Y)=\int_{T_{\lceil \frac{i}{2} \rceil}}\innerproduct{Y(x)}{\eta_i^+}\integrald\hausdorff^1x,
\end{equation*}
for every vector field $Y$ compactly supported in $C_{\lceil\frac{i}{2}\rceil}$.
This contradicts $V\partial E=0$ with a suitable choice of $Y$.

Now, suppose $J=\{i,i+1\}$ at least one of the following must hold: $\mass{X}(\Sigma_{i+1}^+)>0$ or $\mass{X}(\Sigma_{i+1}^-)>0$.
If both are true then Lemma \ref{constancy} implies that $\mass{X}((E\setminus(\Sigma_i^+\cup\Sigma_{i+1}^+\cup\Sigma_{i+1}^-))\cap C_{\lceil \frac{i}{2} \rceil})=0$.
We obtain the same contradiction as above, since $\eta_{i+1}^++\eta_{i+1}^-=0$.
If only one holds, say $\mass{X}(\Sigma_{i+1}^+)>0$, then we have $\mass{X}((E\setminus(\Sigma_i^+\cup\Sigma_{i+1}^+))\cap C_{\lceil \frac{i}{2} \rceil})=0$ and
\begin{equation*}
V\partial E (Y)=\int_{T_{\lceil \frac{i}{2} \rceil}}\innerproduct{Y(x)}{\eta_i^++\eta_{i+1}^+}\integrald\hausdorff^1x,
\end{equation*}
for every vector field $Y$ compactly supported in $C_{\lceil\frac{i}{2}\rceil}$.
Since $\eta_i^++\eta_{i+1}^+\neq 0$ we again obtain a contradiction with $V\partial E=0$ and conclude the proof.
\end{proof}

\begin{lemma}\label{component dichotomy}
Let $X\in\varifolds_2(\real^3)$ be a component of $V$ and $i\in\{1,\ldots,6\}$.
Suppose $\mass{X}(\Sigma_i)>0$, then either:
\begin{enumerate}[(i)]
\item $i$ is odd and $\mass{X}(\Sigma_j)>0$ for all $j\in\{1,3,5\}$ or
\item $i$ is even and $\mass{X}(\Sigma_j)>0$ for all $j\in\{2,4,6\}$.
\end{enumerate}
\end{lemma}
\begin{proof}
Let $E\subset\real^3$ be the $\mass{V}+\mass{\delta V}$-measurable set with $V\partial E=0$ that defines $X$.

Without loss of generality we may assume $\mass{X}(\Sigma_1)>0$.
It follows from Lemma \ref{plus implies minus} that $\Sigma_1\subset E$.

Define $J=\{j\in\{1,\ldots,6\}:\mass{X}(\Sigma_j)>0\}$ so that $1\in J$ by assumption and note that Lemma \ref{plus implies minus} implies $\mass{X}(E\setminus \cup_{j\in J}\Sigma_j)=0$.
Thus, from Corollary \ref{distributional boundary total} we have
\begin{equation*}
0=V\partial E(Y)=\int_{L^+}\innerproduct{Y(x)}{\sum_{j\in J}\nu_j^+}\integrald\hausdorff^1x+\int_{L^-}\innerproduct{Y(x)}{\sum_{j\in J}\nu_j^-}\integrald\hausdorff^1x,
\end{equation*}
for all $Y\in\smoothcompact{\real^3}{\real^3}$.
Therefore $\sum_{j\in J}\nu_j^+=\sum_{j\in J}\nu_j^-=0$ and we may assume that $J\neq\{1\}$, otherwise we obtain a contradiction.

Now, suppose by contradiction that at least one of the following happens: $3\not\in J$ or $5\not\in J$.
We will consider every possible configuration of $J$ and obtain a contradiction in each case.

\begin{claim}
We must have $\mass{X}(\Sigma_6^+)>0$ and $\mass{X}(\Sigma_4^-)>0$.
\end{claim}
In fact, suppose $\mass{X}(\Sigma_6^+)=0$, so Lemma \ref{plus implies minus} implies $6\not\in J$.
If neither $3,5\not\in J$ then the only possibilities left for $J$ are $\{1,2\}$, $\{1,4\}$ or $\{1,2,4\}$.
Since $\innerproduct{\nu_1^+}{\nu_k^+}>0$ for $k=2,4$ and $\nu_2^++\nu_4^+\neq 0$ it contradicts $\sum_{j\in J}\nu_j^+\neq 0$.
Similarly suppose $3\in J$ and $5\not\in J$, so the possibilities are $\{1,3,2\}$, $\{1,3,4\}$ or $\{1,3,2,4\}$ and note:
\begin{itemize}
\item if $2\in J$ then we have a contradiction since $\nu_2^++\nu_3^+=0$ and $\nu_1^++\nu_4^+\neq 0$;
\item if $2\not\in J$ then we also produce a contradiction from $\innerproduct{\nu_4^+}{\nu_k^+}>0$ for $k=1,3$ and $\nu_1^++\nu_3^+\neq 0$.
\end{itemize}
Alternatively, we may suppose $3\not\in J$ and $5\in J$ to obtain the same contradiction.

Arguing as above but contradicting $\sum_{j\in J}\nu_j^-=0$, we obtain $\mass{X}(\Sigma_4^-)>0$, which concludes the proof of the claim.

It follows from Lemma \ref{plus implies minus} that in fact we must have $4,6\in J$.

\begin{claim}
We must have $\mass{X}(\Sigma_2^+)>0$.
\end{claim}
Suppose by contradiction that $\mass{X}(\Sigma_2^+)=0$, so Lemma \ref{plus implies minus} implies $2\not\in J$.
We are assuming that $\{3,5\}\not\subset J$ and we already know $\{4,6\}\subset J$ so the only remaining possibilities for $J$ are $\{1,4,6\}$, $\{1,3,4,6\}$ or $\{1,5,4,6\}$.
All three cases contradict $\sum_{j\in J}\nu_j^+= 0$, which proves the claim.

Once again, Lemma \ref{plus implies minus} implies $2\in J$. 
Finally, the only remaining possibilities for $J$ are $\{1,2,4,6\}$, $\{1,3,2,4,6\}$ or $\{1,5,2,4,6\}$, all of which contradict $\sum_{j\in J}\nu_j^+= 0$ since $\nu_2^++\nu_4^++\nu_6^+=0$, $\nu_1^++\nu_3^+\neq 0$ and $\nu_1^++\nu_5^+\neq 0$.
This concludes the proof for $i=1$.

The assumption of $i=1$ was arbitrary and a similar proof can be repeated with $i=3$ or $i=5$.
The even case also follows the same argument.
\end{proof}

\begin{remark}\label{z indecomposable}
It follows from the proofs of Lemmas \ref{plus implies minus} and \ref{component dichotomy} that $Z_1$ and $Z_2$ are indecomposable.
Indeed, suppose $X$ is a component of $Z_1$, in which case we may assume without loss of generality that $\mass{X}(\Sigma_1^+)>0$.
Following the same argument as in Lemma \ref{plus implies minus} we have that $\Sigma_1\subset\support{\mass{X}}$.
Similarly, following the arguments of Lemma \ref{component dichotomy}(a) we have that $\Sigma_3\subset\support{\mass{X}}$ and $\Sigma_5\subset\support{\mass{X}}$, which implies that $X=Z_1$ and proves that $Z_1$ is indecomposable.
The same holds for $Z_2$.
Therefore, $\Xi=\{Z_1,Z_2\}$ is a decomposition of $V$.
\end{remark}

\begin{lemma}\label{planar support}
Let $X\in\varifolds_2(\real^3)$ be a component of $V$.
Suppose $X$ is a curvature varifold (without boundary), then the following statements hold for any $i=1,\ldots,6$:

If $\mass{X}(\Sigma_i^\pm)>0$ then $\mass{X}(\Sigma_{j^\pm(i)}^\pm)>0$, where $j^\pm(i)\in\{1,\ldots,6\}$ is such that $\nu_i^\pm+\nu_{j^\pm(i)}^\pm=0$.
\end{lemma}
\begin{proof}
Without loss of generality let us assume that $i=1$ and consider the case $\mass{X}(\Sigma_1^+)>0$, in which case $j^+(1)=6$.

Suppose by contradiction that $\mass{X}(\Sigma_6^+)=0$.
It follows from Lemma \ref{constancy} that $\Sigma_1^+\subset\support{\mass{X}}$ and for each $j=2,3,4,5$ either $\mass{X}(\Sigma_j^+)=0$ or $\Sigma_j^+\subset\support{\mass{X}}$.
We may define the set $J=\{j\in\{1,2,3,4,5\}:\Sigma_j^+\subset\support{\mass{X}}\}$ and note that $1\in J$.
In particular $X\restrict D^+ = \sum_{j\in J}\graphv{\Sigma_j^+}\restrict D^+$.

Let $\phi,\psi:\real\rightarrow\real$ be a arbitrary functions of class $C^1$ with compact support satisfying:
\begin{enumerate}[(a)]
\item $\support{\phi}\subset(1,3)$ and $\support{\psi}\subset(-\frac{1}{2},\frac{1}{2})$;
\item $\int_\real\phi(t)\integrald\lebesgue{1}t=1$ and $\psi(0)=1$.
\end{enumerate}

Observe that $P_1\neq P_j$ for all $j=2,3,4,5$ and take $\varepsilon>0$ such that $d(P_1^+,P_j^+)>2\varepsilon$ for all $j=2,3,4,5$, where the distance is with respect to $\grassmannian{2}{3}$.
Now choose a function $f:\grassmannian{2}{3}\rightarrow\real$ of class $C^1$ with compact support such that $\support{f}\subset\openball{P_{1}^+}{\varepsilon}$ and $f(P_1^+)=1$.
Finally we define $\varphi:\totalgrassmannian{2}{\real^3}\rightarrow\real$ as $\varphi(x,P)=\phi(x_2)\psi(x_1^2+x_3^2)f(P)$ and observe that $\varphi$ is a function of class $C^1$ with compact support and $\support{\varphi}\subset D^+$.
It follows that
\begin{equation*}
\begin{aligned}
\boundary(X,\varphi) = & \sum_{j\in J}\nu_j^+\int_{L^+}\phi(x_2)\psi(0)f(P_j^+)\integrald\hausdorff^1x_2\\
             = & \nu_1^+,
\end{aligned}
\end{equation*}
which contradicts the assumption that $X$ is a curvature varifold (without boundary).
The proof of all other cases are exactly the same.
\end{proof}

\begin{remark}
The statement above is in accordance with the identities described in Remark \ref{vector identities}.
That is, the choice of $j^\pm(i)$ is given by the index corresponding to the vector opposite to $\nu^\pm_i$.
\end{remark}

%
%
%

\begin{theorem}
Let $V=W_1+W_2+W_3+W_4+W_5+W_6$, $Z_1=W_1+W_3+W_5$ and $Z_2=W_2+W_4+W_6$.
The collection $\Xi=\{Z_1,Z_2\}$ is the unique decomposition of $V$.
\end{theorem}
\begin{proof}
Suppose $X\in\varifolds_2(\real^3)$ is a component of $V$ and $X\not\in\{Z_1,Z_2\}$.
Since $X\neq 0$ and $\support{\mass{V}}=\cup_{i=1}^6\bar\Sigma_i$ we must have $\mass{X}(\Sigma_{i_0})>0$ for some $i_0\in\{1,\ldots,6\}$.

Let us first assume that $i_0$ is odd, hence Lemmas \ref{component dichotomy} and \ref{plus implies minus} imply that $\Sigma_i\subset\support{\mass{X}}$ for all $i=1,3,5$.
By assumption $X\neq Z_1$, that is, there exists $j_0\in\{2,4,6\}$ such that $\mass{X}(\Sigma_{j_0})>0$.
Again by Lemmas \ref{component dichotomy} and \ref{plus implies minus} we have $\Sigma_j\subset\support{\mass{X}}$ for all $j=2,4,6$, which implies $X=V$ and contradicts the fact that $X$ is indecomposable.

Had we begun by assuming $i_0$ to be even, we would have obtained the same contradiction, which concludes the proof.
\end{proof}

\begin{corollary}
The varifold $V=W_1+W_2+W_3+W_4+W_5+W_6\in\cvarifolds_2(\real^3)$ is a curvature varifold (without boundary) that does not admit a decomposition by curvature varifolds (without boundary).
\end{corollary}


\end{document}